\documentclass{amsart}

\usepackage{amssymb}
\usepackage{amsmath}
\usepackage{amsthm}
\usepackage{amsfonts, dsfont}
\usepackage[english]{babel}
\usepackage[utf8]{inputenc}
\usepackage{enumerate}
\usepackage{color}
\usepackage[numbers,sort&compress]{natbib}

\newcommand{\cA}{{\mathcal A}}
\newcommand{\cB}{{\mathcal B}}
\newcommand{\cC}{{\mathcal C}}

\newcommand{\cL}{{\mathcal L}}

\newcommand{\cP}{{\mathcal P}}

\newcommand{\cS}{{\mathcal S}}

\newcommand{\C}{{\mathbb C}}

\newcommand{\E}{{\mathbb E}}

\newcommand{\N}{{\mathbb N}}

\renewcommand{\P}{{\mathbb P}}

\newcommand{\R}{{\mathbb R}}
\renewcommand{\S}{{\mathbb S}}

\newcommand{\Z}{{\mathbb Z}}

\newtheorem{thm}{Theorem}[section]
\newtheorem{lem}[thm]{Lemma}
\newtheorem{prop}[thm]{Proposition}
\newtheorem{coro}[thm]{Corollary}
\newtheorem{rem}[thm]{Remark}

\def\la{\langle}
\def\ra{\rangle}

\newcommand\comment[1]{}
\def\summ#1#2#3{\sum_{#1=#2}^{#3}}

\newcommand{\eps}{\varepsilon}
\renewcommand{\Re}{{\rm Re}\,}
\newcommand{\ind}{\mathds1}
\newcommand{\ds}{\displaystyle}
\def\nn#1{\left\|#1\right\|}

\def\inv{^{-1}}

\def\proba{\mathbb P}

\def\pp#1{\left(#1\right)}

\def\ccbb#1{\left\{#1\right\}}

\def\limn{\lim_{n\to\infty}}

\begin{document}\sloppy

\title[Generalized random fields and L\'evy's continuity theorem]{Generalized random fields and \\ L\'evy's continuity theorem on\\ the space of tempered distributions}

\author{Hermine Bierm\'e}
\address
{
Hermine Bierm\'e,
Laboratoire de Math\'ematiques et Applications UMR CNRS 7348,
Universit\'e de Poitiers, 
Boulevard Marie et Pierre Curie
86962 Futuroscope Chasseneuil Cedex, France.
}
\email{hermine.bierme@math.univ-poitiers.fr}

\author{Olivier Durieu}
\address{
Olivier Durieu\\
Laboratoire de Math\'ematiques et Physique Th\'eorique, UMR-CNRS 7350\\
F\'ed\'eration Denis Poisson, FR-CNRS 2964\\
Universit\'e Fran\c{c}ois--Rabelais de Tours, Parc de Grandmont, 37200 Tours, France.
}
\email{olivier.durieu@lmpt.univ-tours.fr}

\author{Yizao Wang}
\address
{
Yizao Wang,
Department of Mathematical Sciences,
University of Cincinnati,
2815 Commons Way, ML--0025,
Cincinnati, OH, 45221-0025.
}
\email{yizao.wang@uc.edu}

\date{\today}


\begin{abstract}
In this note, we recall main properties of generalized random fields and present a proof of the continuity theorem of Paul L\'evy for generalized random fields in the space of tempered distributions.  
This theorem was first proved by \citet{Fer68} in a more general setting.
The aim of this note is to provide a self-contained proof that in particular avoids the abstract theory of nuclear spaces.
\end{abstract}

\maketitle

\section{Introduction}

The notion of generalized random fields has been introduced in the 50's by \citet{Gel55}. 
Such a field is a collection of real random variables that is indexed by a 
class of functions and that satisfies linearity and certain regularity conditions. It can also be seen as a single random variable taking values in the space of generalized functions (or distributions). The name \textit{random distribution} is also used in the literature. 
The interest for generalized random fields was first due to their applications in many areas. They were used in constructive quantum field theory (see \citep{Nel73,DobMin76}), 
in stochastic differential equations in infinite-dimensional spaces (see
\citep{KalXio95}), and in sparse stochastic modeling (see \citep{FagAmiUns14}).
In comparison to classical probability theory, two theorems stand out for generalized random fields. 
The first is a Bochner-type theorem obtained by \citet{Min59} that gives sufficient conditions on a functional on the Schwartz space to be the characteristic functional of some generalized random field.
The second is a L\'evy-type continuity theorem, saying that the convergence in distribution of generalized random fields  can be determined by the pointwise convergence of the corresponding characteristic functionals, provided in addition that the limit is continuous at zero. The L\'evy's continuity theorem was obtained by \citet{Fer68} and generalized by \citet{Mey95} (see also \citep{Bou73}) in the setting of nuclear spaces, and  has found various applications in for example particle systems \citep{holley78generalized,GorNua94} and image processing \citep{FagBosUns15}.
Our own motivation came from the investigation of limit theorems for the so-called random-ball models, which will be discussed in a forthcoming paper.
In all these applications, L\'evy's continuity theorem is used for generalized random fields in the space of tempered distributions, which is 
the dual of the Schwartz space.

Although L\'evy's continuity theorem in very general setups has been known since \citet{Fer68} in the 60's, a simple and self-contained proof in 
the specific case of the space of tempered distributions, without using the general framework of nuclear spaces and their dual spaces, cannot be found in the literature.
The space of tempered distributions as the dual of Schwartz space, a countably Hilbert space, enjoys several nice properties. 
Standard references as  \citet{gelfand64generalized} and \citet{KalXio95} provide 
excellent introductions to generalized random fields 
taking values in the space of tempered distributions. However, both treatments stop at Minlos--Bochner's theorem. It  is the purpose of this paper to fill in the gap by presenting a proof of L\'evy's continuity theorem for generalized random fields, without using abstract theory on nuclear spaces.

Throughout the paper, the integer $d\ge1$ is fixed; 
$\N$ stands for the set of nonnegative integers $\{0,1,2,\dots\}$;
for $n\in\N$, $\Gamma_n$ denotes the net-square $\{0,\ldots,n\}^d$; 
for each integer $s\ge2$, $\zeta(s)$ denotes the finite constant $\sum_{n\ge 1} n^{-s}$.

\section{Generalized random fields and main results}\label{sec:2}

Recall that the \textit{Schwartz space} $\cS=\cS(\R^d)$ consists of infinitely differentiable functions $f:\R^d\to \R$ that are rapidly decreasing, that is, for all $m\in\N$ and $j=(j_1,\ldots,j_d)\in\N^d$,
\begin{equation}\label{eq:norm_mj}
\|f\|_{m,j}=\sup_{x \in \R^d}(1+|x|)^m \left|D^{j}f(x)\right| < \infty,
\end{equation}
where $D^jf(x)=\frac{\partial^{j_1}\ldots\partial^{j_d}}{\partial
x_1^{j_1}\ldots \partial x_d^{j_d}}f(x)$ denotes the partial
derivative of order $j$. 
It is a linear vector space and it is equipped
with the topology generated by the family of semi-norms $\|\cdot\|_{m,j}$, $m\in\N$ and $j\in\N^d$.
It is well known that  $\cS$ is a  Fr\'echet space (i.e.\ a locally convex metrizable complete space). 

The \textit{space of tempered distributions} $\cS'=\cS'(\R^d)$ is the topological dual of $\cS$, the space of continuous linear functionals on $\cS$. 
For $F\in\cS'$ and $f\in\cS$, the real number $F(f)$ is usually denoted by $(F,f)$.  The duality bracket $(\,\cdot,\cdot\,):\cS'\times\cS\to \R$ is a bilinear form.

 We shall work with two topologies on $\cS'$, the strong topology $\tau_s$ and the weak topology $\tau_w$.
 The \textit{strong topology}  $\tau_s$ on $\cS'$ 
is generated by the family of semi-norms
\[
 q_B(F)=\sup_{f\in B}|(F,f)|,\quad B\subset \cS \mbox{ bounded.}
\]
Recall that a subset $B$ of a topological vector space is bounded if for all neighborhood $V$ of $0$, there exists $\lambda>0$ such that $B\subset \lambda V$. 
The \textit{weak topology} on $\cS'$ 
is simply induced by the family of semi-norms $|(\,\cdot,\varphi)|$,  $\varphi \in \cS$. In particular $F_n$ converges to $F$ in $\cS'$ with respect to the weak topology, when
$\limn(F_n,\varphi)= (F,\varphi)$ for all $\varphi \in \cS$. 
We show in Section~\ref{sec:sigmafields} that the Borel $\sigma$-fields corresponding to both topologies coincide. Therefore we shall talk about the Borel $\sigma$-fields $\cB(\cS')$ of $\cS'$ without specifying the topology. 

Let $(\Omega,\cA,\P)$ be a probability space. 
By a \textit{generalized random field} defined on $(\Omega,\cA,\P)$, we refer to a random variable $X$ with values in $(\cS',\cB(\cS'))$, namely, a measurable map $X:(\Omega,\cA)\to (\cS',\cB(\cS'))$.  In this case, $X(f)$ is a real-valued random variable for all $f\in\cS$ and $f\in\cS\mapsto X(f)(\omega)$ is linear and continuous for all $\omega\in\Omega$. 
Here, we are mostly interested in the law of the generalized random field
\[
\P_X = \P\circ X\inv,
\]  the induced probability measure on $(\cS', \cB(\cS'))$. We shall nevertheless talk about convergence of generalized random fields instead of convergence of laws of generalized random fields, for the sake of simplicity. 

\medskip

A fundamental object in the investigation of generalized random fields is the characteristic functional, playing the role as the characteristic function for a random variable in $\R$.
First, recall that the law of a real random variable $Y$ defined on $(\Omega,\cA,\P)$ is uniquely determined by its characteristic function $\cL_Y$ defined as
\[
\cL_Y(t)=\E(e^{itY})=\int_{\Omega}e^{itY}d\P,\quad t\in\R.
\]
The function $\cL_Y$ is positive-definite, continuous and satisfies $\cL_Y(0)=1$. 
Here we understand that a complex-valued function $\cL$ is \textit{positive-definite} if for all $m\ge1$, all complex numbers $c_1,\ldots,c_m$, and all $t_1,\ldots,t_m\in \R$,
$ \sum_{i,j=1}^m c_i\bar c_j \cL(t_i-t_j)\ge 0$.
The well known Bochner's theorem 
states that conversely, if a map $\cL:\R\to\C$ 
is positive-definite, continuous at $0$, and $\cL(0)=1$, then there exists a real random variable $Y$ such that $\cL_Y=\cL$. 

Similarly, for a generalized random field $X$ with law $\P_X$, we define its \textit{characteristic functional} by
\begin{equation*}
\cL_X(f)=\E(e^{iX(f)})=\int_{\cS'}e^{iF(f)}d\P_{X}(F),\quad f\in \cS.
\end{equation*}
Note that $\cL_X$ is again positive-definite, continuous, and $\cL_X(0)=1$.
Here and in the sequel, we say that a functional $\cL:\cS\to\C$ is positive-definite, if for all $m\ge 1$, $c_1,\dots,c_m\in\C$ and $f_1,\dots,f_m\in\cS$, 
\[
\summ i1m\sum_{j=1}^m c_i\bar c_j\cL_X(f_i-f_j)\ge 0.
\]
The Minlos--Bochner theorem says that the converse is true.
\begin{thm}[Minlos--Bochner]\label{thm:MinlosGRF}
If a functional $\cL: \cS \rightarrow \C$ is positive-definite, continuous at $0$, and $\cL(0)=1$, there exists a generalized random field $X$ defined on a probability space $(\Omega,{\mathcal A},\mathbb{P})$ such that ${\mathcal L}_X=\cL$.
\end{thm}
In order to be self-contained, we give a proof of this theorem in Section \ref{sec:Milnos} as well as a similar result for random variables taking values in $\cS$ (Theorem~\ref{thm:MinlosS}), which will be needed in the proof of L\'evy's continuity theorem.
Minlos--Bochner's theorem has a direct application in proving existence of generalized random fields. 
For example, consider the functional $\cL:\cS\to\C$ defined by $\cL(f)=\exp(-{\|f\|_2^2}/2)$, where $\|f\|_2$ denotes the $L^2$-norm of $f$. Theorem~\ref{thm:MinlosGRF} yields the existence of a generalized random field $X$ such that $\cL_X(\cdot)=\cL(\cdot)$. This field $X$ is 
known as the standard generalized Gaussian field in the literature (see e.g.~\citep{Kuo96,ito83}).

More generally, there are situations that  a collection of real random variables $X=(X(f))_{f\in \cS}$ are given, such that for all $m\ge 1$, $a_1,\dots,a_m\in\R, f_1,\dots,f_m\in\cS$, 
\begin{equation}\label{eq:linear}
X\pp{\summ i1ma_if_i} = \summ i1m a_iX(f_i), \mbox{ a.s.}
\end{equation}
Although a priori this does not necessarily imply that $(X(f))_{f\in\cS}$ is determined by a generalized random field $X$, Minlos--Bochner's theorem yields a simple criterion for $X$ to admit a version as a generalized random field.
By a version of $X$, we mean a generalized random field $\tilde X$, not necessarily defined on the same probability space as $X$, such that $(X(f))_{f\in\cS}$ and $(\tilde X(f))_{f\in\cS}$ have the same finite-dimensional distributions. Given a collection of random variables $(X(f))_{f\in\cS}$, with slightly abused notation we write ${\mathcal L}_X(f)=\mathbb{E}(e^{iX(f)})$. The linearity~\eqref{eq:linear} shall imply that $\cL_X(f)$ is positive-definite. We hence have the following.

\begin{coro}\label{cor:MinlosGRF}
Let $X=(X(f))_{f\in \cS}$ be a collection of real random variables on $(\Omega,{\mathcal A},\mathbb{P})$ that are linear in the sense of~\eqref{eq:linear}.
If ${\mathcal L}_X(f)$ is continuous at zero on $\cS$, then $X$ admits a version that is a generalized random field.
\end{coro}
Our notion of version is  weaker than the notion of regularization in \citep{ito83} where $\tilde X$ is required to be on the same probability space, and $\proba(\tilde X(f) = X(f)) = 1$ for all $f\in\cS$. However, our notion of version suffices for many applications on convergence in distribution of generalized random fields.

Another important theorem concerning characteristic functional is the L\'evy's continuity theorem that characterizes convergence in distribution. 
Let $(X_n)_{n\ge 1}$ and $X$ be generalized random fields 
with laws $(\P_{X_n})_{n\ge 1}$ and $\P_X$ respectively.
We say that $X_n$ converges in distribution to $X$ with respect to the strong topology (of $\cS'$) if 
\begin{equation}\label{eq:convergence}
\limn\int_{\cS'} \varphi(F) d\P_{X_n}(F)= \int_{\cS'} \varphi(F) d\P_X(F),\mbox{ for all $\varphi\in\cC_b(\cS',\tau_s)$,}
\end{equation}
where $\cC_b(\cS',\tau_s)$ is the space of bounded continuous functions on $\cS'$ given the strong topology.

The main result of the paper is a self-contained proof of the following version of L\'evy's continuity theorem for generalized random fields.
\begin{thm}\label{thm:LevyGRF}
Let $(X_n)_{n\ge 1}$ be a sequence of generalized random fields.
If $\cL_{X_n}$ converges pointwise to a functional $\cL:\cS\to \C$ which is continuous at $0$, then there exists a generalized random field $X$ such that $\cL_X=\cL$ and $X_n$ converges in distribution to $X$ with respect to the strong topology.
\end{thm}
Another natural notion of convergence in the space $\cS'$ is the convergence in distribution with respect to the weak topology, which is defined similarly as in~\eqref{eq:convergence}, with test functions from $\cC_b(\cS',\tau_w)$ instead of $\cC_b(\cS',\tau_s)$. It is straightforward to show that the convergence with respect to the strong topology implies the convergence with respect to the weak topology, and that the latter convergence implies the convergence of the characteristic functionals. The main theorem above shows that the three types of convergence are equivalent.

\begin{coro}\label{coro:Levy}
Let $(X_n)_{n\ge 1}$, $X$ be generalized random fields. The following conditions are equivalent:
\begin{enumerate}[(i)]
\item $X_n$ converges in distribution to $X$ with respect to the strong topology,
\item $X_n$ converges in distribution to $X$ with respect to the weak topology,
\item $\limn\cL_{X_n}(f)= \cL_X(f)$ for all $f\in\cS$,
\item $X_n(f)$ converges in distribution to $X(f)$ in $\R$ for all $f\in\cS$.
\end{enumerate}
\end{coro}

As we pointed out before, Theorem~\ref{thm:LevyGRF} and the discussion above are a special case of a more general result of \citep{Fer68,Mey95}, where, however, the background on nuclear spaces are required. The aim of the paper is to present a simpler and self-contained proof in this special case of the space $\cS'$ that cannot be found in the literature, to the best of our knowledge.

\medskip

The paper is organized as follows. Some preliminary results on Schwartz space and the space of tempered distributions are provided in Section~\ref{sec:spaces}. Section~\ref{sec:Milnos} establishes Milnos--Bochner's theorems. Section~\ref{sec:Levy} proves the L\'evy's continuity theorem for generalized random fields. Section~\ref{sec:appendix} provides some topological properties of $\cS$ and $\cS'$ that are needed in the proofs.

%%%%%%%%%%%%%%%%%%%%%%%%%%%%%%%%%%%%%%%%%%%%%%%%%%%%%%%%%%%%%%%%%%%%%%%%%%%%%%

\section{The Schwartz space and the space of tempered distributions}\label{sec:spaces}

\subsection{The Schwartz space $\cS$ as a countably Hilbert space}

A countable orthonormal basis (a complete orthonormal system) of $L^2(\R)$ is given by the family of Hermite functions $\{h_n\}_{n\in\N}$, defined as
\[
 h_n(x)=(-1)^n(2^n n! \sqrt{\pi})^{-1/2} e^{\frac{x^2}{2}} \left(\frac{d}{dx}\right)^ne^{-x^2}, \quad n\in\N, x\in\R.
\]
On $\R^d$, we define the Hermite functions $\{h_n\}_{n\in\N^d}$ by
\[
 h_n(x)=\prod_{i=1}^d h_{n_i}(x_i),\quad n=(n_1,\ldots,n_d)\in\N^d,\, x=(x_1,\ldots,x_d)\in\R^d.
\]
The family of Hermite functions on $\R^d$ is an orthonormal basis of $L^2(\R^d)$. 
Any function $f\in L^2(\R^d)$ can be written as
\[
f=\sum_{n\in\N^d}\langle f,h_n \rangle h_n,
\]
where $\langle\cdot,\cdot\rangle$ denotes the inner product in $L^2(\R^d)$ given by $\langle f,g\rangle=\int_{\R^d} f(x) g(x) dx$.
We can then define a family of increasing norms on $L^2(\R^d)$ by
\begin{equation}\label{eq:nn_p}
\|f\|_p^2=\sum_{n\in\N^d}(1+n)^{2p} \langle f,h_n \rangle^2,\quad p\in\N, \, f\in L^2(\R^d),
\end{equation}
where here and in the sequel, we write
\begin{equation}\label{eq:1+n}
(1+n)^{q}=\prod_{i=1}^d(1+n_i)^{q}\quad\mbox{ for }\quad q\in\Z, n=(n_1,\ldots,n_d)\in\N^d.
\end{equation}
For each $p\in\N$, we denote by $\cS_p=\cS_p(\R^d)$ the space of all $L^2(\R^d)$-functions such that $\|f\|_p<\infty$. Note that each $\|\cdot\|_p$ is the norm associated to the inner product
\[
  \langle f,g\rangle_p
  = \sum_{n\in\N^d}(1+n)^{2p} \langle f,h_n \rangle \langle g,h_n \rangle.
\]
The spaces $(\cS_p)_{p\in\N}$ play a crucial role in our study.

For each $p\in\N$, $\cS_p$ is a separable Hilbert space with a countable orthonormal basis given by $\{h_n^{(p)}:=(n+1)^{-p}h_n\,,\,n\in\N^d\}$.
Note that $\cS_0$ is just the space $L^2(\R^d)$, and for each $p\in\N$, $\cS_{p+1}\subset\cS_p$ with $\|\cdot\|_p\le\|\cdot\|_{p+1}$.
It is well known that $\cS=\bigcap_{p\ge 0}\cS_p$ and that both families of 
semi-norms $\{\|\cdot\|_{m,j}\mid m\in\N, j\in\N^d\}$ and $\{\|\cdot\|_{p}\mid p\in\N\}$ generate the same topology on $\cS$. See for example \citep[Theorem 1.3.2]{KalXio95}. A basis of neighborhoods of $0$ in $\cS$ is then given by the sets
\begin{equation}\label{eq:B_p}
 B_p(r)=\{f\in\cS \mid \|f\|_p<r\},\quad p\in\N,r>0.
\end{equation}
Concerning the space of tempered distributions, we have $\cS'=\bigcup_{p\ge 0}\cS_p'$, where $\cS_p'$ is the dual of $\cS_p$. 
On $\cS_p'$, we consider the dual norm defined by 
\[
\|F\|_p'=\sup_{\|f\|_p\le 1}|(F,f)| \mbox{ for all } F\in\cS_p'. 
\]

\subsection{The multi-sequence representation of $\cS$ and $\cS'$}

The representation of $\cS$ as a countably Hilbert space is 
convenient
 to establish the main theorems in Sections~\ref{sec:Milnos}~and~\ref{sec:Levy}.
As soon as we have this representation, we can think of the spaces $\cS$ and $\cS'$ as subspaces of $\R^{\N^d}$, the space of real-valued multi-sequences. It gives a more concrete description of these spaces and it simplifies our arguments in the sequel.
In this section we recall this multi-sequence representation as  introduced by \citet{ReeSim80}.

\medskip

Let $\R^{\N^d}$ be the set of multi-sequences with real values. For $a=(a_n)_{n\in\N^d}$ and $b=(b_n)_{n\in\N^d}$ such that $\sum_{n\in\N^d}|a_n b_n|<\infty$, we denote 
\[
 \langle a,b \rangle=\sum_{n\in\N^d}a_n b_n.
\]
For all integer $p\in\Z$, and all multi-sequence $a=(a_n)_{n\in\N^d}\in \R^{\N^d}$, we set
\[
 \|a\|_p^2=\sum_{n\in\N^d}(1+n)^{2p}a_n^2.
\]
We define 
$$
\S_p=\{a\in \R^{\N^d} \mid  \|a\|_p<\infty\}, \quad p\in\Z,\;\text{ and }\;\S=\bigcap_{p\in\N}\S_p.
$$
Note that for each $p\in\Z$, $\S_p$ is a Hilbert space with respect to the inner product $ \langle a,b \rangle_p=\sum_{n\in\N^d}(1+n)^{2p}a_n b_n$, and $\|a\|_p^2 = \langle a,a\rangle_p$. The set $\S$ is then given the topology generated by the family of Hilbert
norms $\|\cdot\|_p$, $p\in\Z$.

\medskip

For each $p\in\Z$, we let $\S_p'$ denote the dual space of the space $\S_p$. The dual norm on $\S_p'$ is given by $\|F\|_p'=\sup_{\|a\|_p\le 1}|(F,a)|$, and we have $\S'=\bigcup_{p\ge 0}\S_p'$. 
\begin{lem}\label{lem:isom}
For each $p\in\Z$, the map
$b\mapsto \langle b,\cdot \rangle$
is an isometry from $\S_{-p}$ to $\S_{p}'$. 
\end{lem}
\begin{proof}
Let $\phi$  denote the map of interest $b\mapsto \langle b,\cdot\rangle$.
Note that, by Cauchy--Schwarz inequality, for $a\in\S_p$ and $b\in\S_{-p}$,
\[
 |\langle b,a\rangle|\le \|a\|_p\|b\|_{-p},
\]
and $\|\phi(b)\|_p'=\sup_{\|a\|_p\le 1}|\langle b,a\rangle|\le \|b\|_{-p}$. Thus $\phi$ maps $\S_{-p}$ into $\S_p'$. Further, for $b\in\S_{-p}$ by setting $a=(a_n)_{n\in\N^d}$ with $a_n=\|b\|_{-p}^{-1}(n+1)^{-2p}b_n$ we have $\|a\|_p=1$ and thus
\[
 \|\phi(b)\|_p'\ge |\langle b,a\rangle|=\|b\|_{-p}^{-1}\sum_{n\in\N^d}(n+1)^{-2p}b_n^2= \|b\|_{-p}.
\]
Thus, for all $b\in\S_{-p}$, $\|\phi(b)\|_p'=\|b\|_{-p}$.

Now, to see that $\phi$ is surjective, suppose $F\in\S_p'$, $F\ne 0$, let $c$ be an
element in $\S_p$ such that $\nn c_p = 1$ and $c$ is orthogonal to ${\rm Ker}(F)$. Define $b = (b_n)_{n\in\N^d}$ with 
$b_n = F(c)c_n(1+n)^{2p}$ for $n\in\N^d$. Then, it follows that $\nn b_{-p} = F(c)<\infty$, hence $b\in\S_{-p}$. Now for all $a\in\S_p$, write $a = a_1+a_2$ with $a_1 = (F(a)/F(c))c$ and $a_2 = a-a_1\in{\rm Ker}(F)$. Then $\la b,a\ra = \la b,a_1\ra = F(a)$. 
\end{proof}
In the sequel, we identify the spaces $\S_p'$ and $\S_{-p}$. We thus have $\S'=\bigcup_{p\ge 0}\S_{-p}\subset \R^{\N^d}$ and then the duality bracket $(\,\cdot,\cdot\,)$ corresponds to $\langle\,\cdot,\cdot\rangle$. 
\begin{lem}\label{lem:S'}
Let $b=(b_n)_{n\in\N^d}\in\R^{\N^d}$. Then $b\in\S'$ if and only if there exist $p\in\N$ and $c>0$ such that for all $n\in\N^d$, $|b_n|\le c(1+n)^p$. 
\end{lem}
\begin{proof}
If $b\in\S'$, there exists $p\in\N$ such that $b\in\S_{-p}$. Then, for all $n\in\N^d$, $(1+n)^{-p}|b_n|\le \|b\|_{-p}$ and the first implication, with $c=\|b\|_{-p}$, is proved.

\medskip

\noindent Conversely, if there exist $p\in\N$ and $c>0$ such that for all $n\in\N^d$, $|b_n|\le c(1+n)^p$, then
\[
 \|b\|_{-p-1}^2=\sum_{n\in\N^d}(n+1)^{-2p-2}b_n^2 \le c^2 \sum_{n\in\N^d}(n+1)^{-2}=c^2\zeta(2)<\infty.
\]
Thus $b\in\S_{p+1}'\subset \S'$.
\end{proof}

\medskip

Now, we show that $\S$ is  topologically isomorphic to $\cS$.
Recall that two topological vector spaces are said to be \textit{topologically isomorphic} if there exists a linear isomorphism from one to the other which is also a homeomorphism (a bi-continuous map). Such a map is called a \textit{topological isomorphism}.

\begin{prop}\label{prop:homeoS}
The map $\Phi: \cS\to \S$, $f\mapsto (\langle f, h_n\rangle)_{n\in\N^d}$ is a topological isomorphism. 
\end{prop}
\begin{proof}
For any $f\in\cS$ and $p\in\N$, 
\[
 \|\Phi(f)\|_p^2=\sum_{n\in\N^d}(1+n)^{2p}\langle f,h_n\rangle^2=\|f\|_p^2<\infty.
\]
Thus the linear map $\Phi$ is well-defined and continuous. Further, each $\|\cdot\|_p$ being a norm on $\cS$, the map $\Phi$ is injective.
Let us show that it is also surjective. Let $a=(a_n)_{n\in\N^d}\in\S$ and define $f_m=\sum_{n\in\Gamma_m}a_n h_n$, for all $m\in\N$. 
For all $p\in\N$ and $k\in\N$, 
\[
 \|f_{m+k}-f_m\|_p^2\le\sum_{n\notin\Gamma_{m}}(n+1)^{2p}a_n^2\to 0\;\text{ as }m\to\infty.
\]
Thus $(f_m)_{m\in\N}$ is a Cauchy sequence in $\cS$ and then it converges to $f\in\cS$.
By continuity of $\Phi$, we see that $\Phi(f)=a$.
The equalities $\|\Phi(f)\|_p=\|f\|_p$, $p\in\N$, prove that $\Phi$ is a homeomorphism.
\end{proof}
We can also introduce a topology on $\S'$ to make it topologically isomorphic to $(\cS',\tau_s)$, thanks to the following result.
\begin{prop}\label{prop:homeoS'}
The map $\Psi: \cS'\to \S'$, $F\mapsto ((F, h_n))_{n\in\N^d}$ is a linear isomorphism.
\end{prop}
\begin{proof}
For all $F\in\cS'$ and $n\in\N^d$, there exists $p\in\N$, such that $F\in\cS_p'$ and then, for all $n\in\N^d$, 
\[
|(F,h_n)|\le \|F\|_{p}'\|h_n\|_p=\|F\|_p'(n+1)^{p}.
\]
Thus, by Lemma~\ref{lem:S'}, the linear map $\Psi$ is well-defined.
Further, $\Psi$ is injective, since $\Psi(F)=0$ implies that $(F,h_n)=0$ for all $n\in\N^d$, and then by linearity of $F$, it implies $F(f)=0$ for all $f\in\cS$.
It remains to show that $\Psi$ is surjective. Let $b=(b_n)_{n\in\N^d}\in\S'$ and hence, $b\in\S_{-p}$ for some $p\in\N$. We define the map $F:\cS\to\R$, $f\mapsto\langle b, \Phi(f)\rangle$, where $\Phi$ is the 
topological isomorphism in Proposition~\ref{prop:homeoS}.
Then, $F$ is a continuous linear map from $\cS$ to $\R$. Indeed, the linearity is clear and the continuity comes from  $|\langle b, \Phi(f)\rangle|\le\|b\|_{-p}\|\Phi(f)\|_p= \|b\|_{-p}\|f\|_p $. Thus $F\in\cS'$. It is clear that $ \Psi(F)=b$.
\end{proof}
As a consequence, $(\cS',\tau_s)$ and $(\S',\Psi(\tau_s))$ are topologically isomorphic.
The topology $\pi_s=\Psi(\tau_s)$ on $\S'$ is called the strong topology on $\S'$. It is generated by the semi-norms 
$$
q_B(\cdot)=\sup_{a\in B}|\langle\,\cdot,a\rangle|, \quad B\text{ bounded in }\S.
$$
In the same way, $(\cS',\tau_w)$ and $(\S',\Psi(\tau_w))$ are topologically isomorphic. 
The topology $\pi_w=\Psi(\tau_w)$ on $\S'$ is called the weak topology on $\S'$. It is generated by the semi-norms
$|\langle \,\cdot,a\rangle|$, $a\in\S$.
In the sequel, if we 
do not 
precise, $\S'$ is assumed given the strong topology.

The following facts will be useful in the next section.
\begin{prop}
 $\S$, $\S'$, and hence $\cS$ and $\cS'$ are separable spaces.
\end{prop}
\begin{proof}
The set of elements of $\R^{\N^d}$ with rational coordinates and only finitely many non-zero ones is dense in $\S_p$ for all $p\in\Z$. It is then dense in both $\S$ and $\S'$.
\end{proof}

\begin{lem}\label{lem:cont}
 For each $p\in\N$, the injection map 
 $i_p:\S_p'\to\S'$ is continuous with respect to the strong topology on $\S'$ (and hence also with respect to the weak topology). 
\end{lem}
\begin{proof}
Let $\S'$ be given the strong topology. Let $V=\{b\in\S'\mid q_B(b)<\eps\}$ for some  bounded set $B\subset\S$ and some $\eps>0$.
For each $b_0\in i_p^{-1}(V)$, we have $q_B(b_0)<\eps$ and for all $b\in\S_p'$,
\begin{align*}
q_B(b)
&\le \sup_{a\in B}|\langle b-b_0,a\rangle| + \sup_{a\in B}|\langle b_0,b\rangle|\\
&\le \|b-b_0\|_p' \sup_{a\in B}\|a\|_p + q_B(b_0).
\end{align*}
Thus, if $\|b-b_0\|_p'< \frac{\eps-q_B(b_0)}{\sup_{a\in B}\|a\|_p +1}$, then $b\in i_p^{-1}(V)$. 
\end{proof}

\subsection{Cylinder $\sigma$-fields and Borel $\sigma$-fields}\label{sec:sigmafields}

On the space $\R^{\N^d}$, 
the product $\sigma$-field generated by the sets
$
 \{a\in\R^{\N^d}\mid (a_i)_{i\in \Gamma}
 \subset
  B\}
$
for all finite subsets $\Gamma$ of $\N^d$ and all Borel sets $B\in\cB(\R^{\#\Gamma})$
is the same as the Borel $\sigma$-field associated to the product topology on $\R^{\N^d}$ (see \citep[Lemma 1.2]{Kal97}). It is denoted by $\cB(\R^{\N^d})$.

On $\S$, the cylinder $\sigma$-field, denoted by $\cC(\S)$, is generated by the cylinder sets
\[
\{a\in\S\mid(\langle b_1,a\rangle,\ldots,\langle b_m,a\rangle )\in B\},
\]
for all $m\ge1$, $b_1,\dots,b_m\in\S'$, $B\in\cB(\R^m)$. We easily see that $\cC(\S)$ is the $\sigma$-field induced by $\cB(\R^{\N^d})$ on $\S$.

On $\S'$, the cylinder $\sigma$-field, denoted by $\cC(\S')$, is generated by the cylinder sets
\[
\ccbb{b\in\S'\mid (\langle b,a_1\rangle,\dots,\langle b,a_m\rangle) \in B},
\]
for all $m\ge1$, $a_1,\dots,a_m\in\S$, $B\in\cB(\R^m)$. Again, $\cC(\S')$ is the $\sigma$-field induced by $\cB(\R^{\N^d})$ on $\S'$.
Further, by definition, we see that $\cC(\S')$ is also the Borel $\sigma$-field corresponding to the weak topology $\pi_w$ on $\S'$.

Recall that, in our setting, $\S$ is given the topology generated by the family of norms $\|\cdot\|_p$, $p\in\N$, and $\S'$ is given the strong topology $\pi_s$.
The corresponding Borel $\sigma$-fields are denoted by $\cB(\S)$ and $\cB(\S')$, respectively. Of course, we have
\[
\Phi^{-1}(\cB(\S))=\cB(\cS)\;\text{ and }\;
\Psi^{-1}(\cB(\S'))=\cB(\cS')
\]
for the topological isomorphisms $\Phi$ and $\Psi$ from Propositions~\ref{prop:homeoS} and \ref{prop:homeoS'}, and where $\cB(\cS)$ and $\cB(\cS')$ are respectively the Borel $\sigma$-fields of $\cS$ and $\cS'$ given the strong topology $\tau_s$.

\begin{prop}
$\cC(\S)=\cB(\S)$.
\end{prop}
\begin{proof}
Recall that a countable basis of neighborhoods of $0\in\S$ is 
given by
 the sets $B_p(1/k)$, $p\in\N$, $k\ge1$ (see \eqref{eq:B_p}).
Note that, by Lemma~\ref{lem:isom}, $B_p(1/k)$ is an open set that can be written as 
\begin{align*}
B_p(1/k)&=\left\{a\in\S \mid \sup_{\|b\|_{-p}\le 1}|\langle b,a\rangle|<1/k\right\}\\
&=\bigcup_{n\ge1}\bigcap_{b\in D_{-p}}\{a\in \S\mid|\langle b,a\rangle|< 1/k-1/n\},
\end{align*}
where $D_{-p}$ is a countable dense subset in the unit ball of $\S_{-p}$.
So, each $B_p(1/k)$ belongs to $\cC(\S)$.
Now, since $\S$ is separable, there exists a countable dense subset $D\subset \S$, and
for any open set $U$ in $\S$, we have
\[
 U=\bigcup_{\substack{a\in D,\, p\in\N,\, k\ge1\\ \text{ s.t. } a+B_p(1/k)\subset U}}(a+B_p(1/k))
\]
which is a countable union of elements of $\cC(\S)$, and thus $\cB(\S)\subset \cC(\S)$.

\medskip

\noindent The converse inclusion comes from the fact that the maps 
$f_b:a\in \S \mapsto\langle b,a \rangle$, $b\in\S'$, are continuous 
on $\S$ and thus $\cB(\S)$-measurable.
\end{proof}

\begin{prop}\label{prop:sigma-fieldS'}
$\cC(\S')=\cB(\S')$.
\end{prop}
\begin{proof}
Since for all $a\in\S$, the map 
$g_a:b\in \S' \mapsto\langle b,a \rangle$ is continuous 
on $\S'$, the inclusion $\cC(\S')\subset\cB(\S')$ is clear.

\medskip

\noindent For the converse, let $V$ be an open set in $\S'$ with respect to $\pi_s$. By Lemma~\ref{lem:cont}, for each $p\in\N$, $V_p=V\cap \S_{p}'$ is open in $\S_{p}'$, and we can write
$V=\bigcup_{p\in\N}V_{p}$.  To conclude we  prove that open sets in $\S_p'$ (with respect to the topology generated by $\nn\cdot_p'$) belong to  $\cC(\S)$.
Indeed,  for each $p\in\N$ and $r>0$, consider
\[B_{-p}(r)=\{b\in\S' \mid \| b\|_{-p}<r\}=\bigcup_{n\ge1}\bigcap_{a\in D_{p}}\{b\in\S' \mid |\langle b,a\rangle| < r - 1/n\}
\]
where $D_p$ a countable dense subset of the unit ball of $\S_p$, and we see that $B_{-p}(r)\in\cC(\S')$.
A countable basis of the topology of $\S_{p}'$ is given by the sets $b+B_{-p}(1/k)$, for all $k\ge1$
 and $b$ in a given countable dense subset of $\S_{p}'$.
We infer that $V_p\in\cC(\S)$.
\end{proof}

Further, since the cylinder $\sigma$-field on $\S'$ corresponds to the Borel $\sigma$-field of the weak topology, we have the following corollary.
\begin{coro}\label{cor:sigmafields}
On $\S'$ (and $\cS'$), the Borel $\sigma$-field of the strong topology and the Borel $\sigma$-field of the weak topology coincide. 
\end{coro}

%%%%%%%%%%%%%%%%%%%%%%%%%%%%%%%%%%%%%%%%%%%%%%%%%%%%%%%

\section{Minlos--Bochner's theorems}\label{sec:Milnos}

In this section, we study the existence of generalized random fields. 
Thanks to Proposition~\ref{prop:homeoS'}, in order to define a random variable with value in  $(\cS',\cB(\cS'))$, we can define a random variable $X$  with values in $(\S',\cB(\S'))$ and then consider $\Psi(X)$. 
Conversely, for any generalized random field $Y$ with values in $(\cS',\cB(\cS'))$, the variable $X=\Psi^{-1}(Y)$ is a random variable with values in $(\S',\cB(\S'))$. By extension, we also call $X$ a generalized random field and we have
\begin{itemize}
\item for all $\omega \in \Omega$, $X(\omega)\in \S'$;
\item for all $a\in\S$, $ \la X,a\ra:\omega \mapsto \la X(\omega),a\ra $ is a real random variable.
\end{itemize}
We see that a generalized random field $X$ provides a collection $( \la X,a\ra)_{a\in \S}$ of real random variables. 
The finite-dimensional distributions of the generalized random field $X$ on $(\Omega,{\mathcal A},\mathbb{P})$ are given by
the law of the real vectors $(\la X, a_1\ra,\ldots,\la X, a_m\ra)$ for all $m\ge 1$, $a_1,\ldots,a_m\in \cS$. By linearity, the finite-dimensional distributions are determined by the characteristic functional of $X$ defined by 
$$
{\mathcal L}_X(a)=\E(e^{i\la X,a\ra})=\int_\Omega e^{i\la X,a\ra}d\P,\quad a\in \S.
$$
More generally, for a probability measure $\mu$ on $(\R^{\N^d},\cB(\R^{\N^d}))$, we call the characteristic functional of $\mu$ the functional
$$
{\mathcal L}_\mu(a)=\int_{\R^{\N^d}} e^{i\la b,a\ra}d\mu(b),
$$
defined for multi-sequences $a$ such that $\la b,a\ra$ is finite for $\mu$-almost all $b\in\R^{\N^d}$. 
It is always defined for multi-sequences $a\in\R^{\N^d}$ with finitely many non-zero coordinates. In particular, for a generalized random field $X$ with law $\mu=\P\circ X\inv$, we have $\cL_X = \cL_\mu$.

\begin{thm}\label{thm:MinlosS'}
If a functional $\cL:\S\to \C$ is positive-definite, continuous at $0$, and $\cL(0)=1$, then it is
the characteristic functional of a generalized random field on $\S'$.
\end{thm}

Theorem \ref{thm:MinlosGRF} and Corollary~\ref{cor:MinlosGRF} are direct consequences.
To prove the present theorem, we need the following key lemma.

\begin{lem}\label{lem:bound}
Let $\mu$ be a Borel probability measure on $\R^{\N^d}$, $p\in\Z$, $c>0$, and $\eps>0$ such that
for all $a\in\R^{\N^d}$ with finitely many non-zero coordinates, 
\[
 1-\Re \cL_\mu(a)\le \eps+ c\|a\|_p^2.
\]
Then for all $q>p$, and $\sigma>0$,
\[
 \int_{\R^{\N^d}}1- \exp\left(-\frac{\sigma^2}{2}\|b\|_{-q}^2\right)d\mu(b)\le \eps + c\sigma^2\zeta(2(q-p))^d.
\]
In particular,
\[
\mu(\S_{-q}) \ge 1-\varepsilon.
\]
\end{lem}
\begin{proof}
Recall the expression of $(1+n)^q$ in~\eqref{eq:1+n}.
For all $\sigma>0$ and $m\in\N$, we consider the Gaussian probability measure $\nu_{m,\sigma}$ on $\R^{\Gamma_m}$ defined by, 
for $a = (a_n)_{n\in\Gamma_m}\in \R^{\Gamma_m}$,
 \[
  d\nu_{m,\sigma}(a)=\bigotimes_{n\in \Gamma_m}(2\pi\sigma^2 (n+1)^{-2q})^{-1/2}\exp\left(-\frac{a_n^2}{2\sigma^2(n+1)^{-2q}}\right) da_n.
 \]
Then $\nu_{m,\sigma}$ satisfies:
\begin{enumerate}[(i)]
  \item $\ds\nu_{m,\sigma}(\R^{\Gamma_m})=1$, 
  \item $\ds\int_{\R^{\Gamma_m}} a_na_m d\nu_{m,\sigma}(a)=\sigma^2 (n+1)^{-2q}\ind_{\{n=m\}}$, and
  \item $\ds\int_{\R^{\Gamma_m}}e^{i\langle b,a \rangle} d\nu_{m,\sigma}(a)=\exp\left(-\frac{\sigma^2}{2}\sum_{n\in \Gamma_m} (n+1)^{-2q} b_n^2\right)$.
\end{enumerate}
By Fubini's theorem, we have
\begin{align*}
&\hspace{-70pt}\int_{\R^{\N^d}}1-\exp\left(-\frac{\sigma^2}{2}\sum_{n\in \Gamma_m} (n+1)^{-2q} b_n^2\right)d\mu(b)\\
&=\int_{\R^{\N^d}}\int_{\R^{\Gamma_m}}1-e^{i\langle b,a \rangle} d\nu_{m,\sigma}(a)d\mu(b)\\
&=\int_{\R^{\Gamma_m}}\int_{\R^{\N^d}}1-e^{i\langle b,a \rangle}d\mu(b) d\nu_{m,\sigma}(a)\\
&=\int_{\R^{\Gamma_m}}1-\cL_\mu(a) d\nu_{m,\sigma}(a),
\end{align*}
which is then a real quantity, and by assumption,
\begin{align*}
\int_{\R^{\Gamma_m}}1-\cL_\mu(a) d\nu_{m,\sigma}(a)
& \le \eps+c \int_{\R^{\Gamma_m}} \|a\|_p^2 d\nu_{m,\sigma}(a)\\
&=\eps +c\sum_{n\in \Gamma_m} (n+1)^{2p}\int_{\R^{\Gamma_m}}a_n^2d\nu_{m,\sigma}(a)\\
&=\eps +c\sigma^2\sum_{n\in \Gamma_m}(n+1)^{-2(q-p)}.
\end{align*}
Since $\sum_{n\in \Gamma_m} (n+1)^{-2(q-p)}\uparrow \zeta(2(q-p))^d$ as $m\to\infty$, by dominated convergence, we get
\[
 \int_{\R^{\N^d}}1- \exp\left(-\frac{\sigma^2}{2}\sum_{n\in \R^{\N^d}}(n+1)^{-2q}b_n^2\right)d\mu(b)\le \eps +c\sigma^2\zeta(2(q-p))^d.
\]
Now, since $\exp(-\sigma^2\nn b_{-q}^2/2) = 0$ for all $b\notin\S_{-q}$, it follows that
\[
 \mu(\S_{-q})\ge  \int_{\R^{\N^d}}\exp\left(-\frac{\sigma^2}{2}\|b\|_{-q}^2\right)d\mu(b) \ge 1-\epsilon-c\sigma^2\zeta(2(q-p))^d.
\]
The last desired statement follows by taking $\sigma\downarrow0$.
\end{proof}

\begin{proof}[Proof of Theorem~\ref{thm:MinlosS'}]
Let $\cL:\S\to \C$ be positive-definite, continuous at $0$, and $\cL(0)=1$. For any finite subset $\Gamma$ of $\N^d$, we can consider the restriction of $\cL$ to $\R^\Gamma\subset\S$ (by completing the multi-sequences of $\R^\Gamma$ by zeros). Thus, by Bochner's theorem, there exists a unique probability measure $\mu_\Gamma$ on the product $\sigma$-field of $\R^\Gamma$, such that, for all $a\in\R^\Gamma$,
 \[
  \cL(a)=\int_{\R^\Gamma}e^{i\langle b,a\rangle}d\mu_\Gamma(b).
 \]
Further, if $\Gamma\subset \Gamma'$ are finite subsets of $\R^{\N^d}$ and $\pi_{\Gamma',\Gamma}$ is the canonical projection from $\R^{\Gamma'}$ to $\R^\Gamma$, we have $\mu_{\Gamma'}\circ \pi_{\Gamma',\Gamma}^{-1}=\mu_\Gamma$. By Kolomogorov's extension theorem \citep[Theorem 5.16]{Kal97} there exists a unique probability measure $\mu$ on the product $\sigma$-field of $\R^{\N^d}$ such that for all finite $\Gamma$, $\mu\circ \pi_\Gamma^{-1}=\mu_\Gamma$, where  $\pi_\Gamma$ is the canonical projection from $\R^{\N^d}$ to $\R^\Gamma$. Since any element $a\in\S$ can be approximated by multi-sequences with only finitely many non-zero coefficients, by dominated convergence 
and the fact that $\cL$ is continuous,
we get that for all $a\in\S$, $\cL(a)=\cL_\mu(a)$.

\medskip

\noindent It remains to show that $\mu(\S')=1$.
Fix $\eps>0$.
The continuity of $\cL$ implies that there exists $p\in\N$ and $\delta>0$,
\[
 1-\Re \cL(a)\le \eps,\quad \text{ for all }a\in\S\text{ such that }\|a\|_p<\delta. 
\]
Thus,
 \[
 1-\Re \cL_\mu(a)\le \eps+ 2\delta^{-2}\|a\|_p^2,\quad \text{ for all }a\in\S. 
\]
By Lemma~\ref{lem:bound} for $q=p+1$, we have $\mu(\S_{-q})\ge 1-\varepsilon$. 
We infer
$\mu(\S')\ge 1-\eps$ for all $\eps>0$, that is $\mu(\S')=1$.
\end{proof}

In order to prove Lévy's continuity theorem in the next section, we need another Minlos--Bochner's theorem. Now, we consider random variables with values in $(\S,\cB(\S))$.
For such a random variable $X$, one can similarly define its characteristic functional by
$$
\cL_X(b)=\cL_\mu(a)=\int_{\S}e^{i\la b,a \ra}d\mu(a),\quad b \in \S',
$$
where $\mu=\P\circ X\inv$ is the law of $X$ on $(\S,\cB(\S))$.

\begin{thm}\label{thm:MinlosS}
If a functional $\cL:\S'\to \C$ is positive-definite, continuous at $0$ (with respect to the topology $\pi_s$), and $\cL(0)=1$, then it is the characteristic functional of a random variable on $(\S,\cB(\S))$.
\end{thm}
\begin{proof}
Let $\cL:\S'\to \C$ be as in the statement.
As we did in the proof of Theorem~\ref{thm:MinlosS'}, using Bochner's theorem and Kolmogorov's extension theorem, we can show that there exists a probability measures $\mu$ on $\cB(\R^{\N^d})$ such that for all $b\in\S'$, $\cL_\mu(b)=\cL(b)$.

\medskip

\noindent Again, it remains to show that $\mu(\S)=1$.
Fix $\eps>0$. Since $\cL$ is continuous on $\S'$, by Lemma~\ref{lem:cont}, its restriction to $\S_{-p}$ is also continuous for each $p\in\N$. Thus, for all $p\in\N$, there exists $\delta_p>0$,
\[
 1-\Re \cL(b)\le \eps,\quad \text{ for all }b\in\S_{-p}\text{ such that }\|b\|_{-p}<\delta_p. 
\]
Thus,
\[
 1-\Re \cL_\mu(b)\le \eps+ 2\delta_p^{-2}\|b\|_{-p}^2,\quad \text{ for all }b\in\S_{-p}. 
\]
By Lemma~\ref{lem:bound} for $q=-p+1$,  $\mu(\S_{p-1})\ge 1-\eps$. Since this is true for all $p\ge 1$, and since the sets $(\S_p)_{p\in\N}$ form a non-increasing sequence, we get 
\[
\mu(\S)=\mu\left(\bigcap_{p\ge0}\S_p \right)\ge 1-\eps.
\] 
Finally, $\eps>0$ being arbitrary, $\mu(\S)=1$.
\end{proof}

\section{L\'evy's continuity theorem}\label{sec:Levy}

For   Borel probability measures   $(\mu_n)_{n\ge 1}$ and $\mu$ on $(\S',\cB(\S'))$, we say that $\mu_n$ converges weakly  to $\mu$ with respect to the strong topology (of $\S'$), if 
$$
\limn\int_{\S'} \varphi(b) d\mu_n(b)=\int_{\S'} \varphi(b) d\mu(b),\quad \text{ for all }\varphi\in\cC_b(\S',\pi_s).
$$
Accordingly, for generalized random fields $(X_n)_{n\ge1}$ and $X$ in $\S'$, we say $X_n$ converges in distribution to $X$ with respect to the strong topology (of $\S'$), if the corresponding probability law converges weakly. 
Thanks to  Proposition \ref{prop:homeoS'}, to prove  Theorem~\ref{thm:LevyGRF} it is equivalent to prove the following.

\begin{thm}\label{thm:levy}
If $(\mu_n)_{n\ge 1}$ is a sequence of Borel probability measures on $\S'$ such that $\cL_{\mu_n}$ converges pointwise to a functional $\cL:\S\to \C$ which is continuous at $0$, then there exists a Borel probability measure $\mu$ on $\S'$ such that $\cL_\mu=\cL$ and 
 $\mu_n$ converges weakly to $\mu$ with respect to the strong topology.
\end{thm}
\begin{proof}
Let $(\mu_n)_{n\ge 1}$ be as in the statement.
By pointwise convergence, $\cL$ is also positive-definite and $\cL(0)=1$. By Theorem~\ref{thm:MinlosS'},
there exists a Borel probability measure $\mu$ on $\S'$ such that $\cL_\mu=\cL$.

\medskip

\noindent Let $\cP(\S')$ denote the set of all Borel probability measures on $\S'$. The main step of the proof is to show that, under the assumptions, the sequence $(\mu_n)_{n\ge 1}$ is tight in $\cP(\S')$, i.e.\ for all $\eps>0$, there exists a strongly compact set $K$ in $\S'$ such that $\mu_n(K)\ge 1-\eps$ for all $n\ge 1$. 

\medskip

\noindent Assume we have proved it and let us conclude the proof. From tightness, we infer that the sequence $(\mu_n)_{n\ge 1}$ is relatively sequentially 
compact with respect to the topology of the weak convergence in $\cP(\S')$, see \citep[Theorem 2.2.1]{KalXio95}.
Thus, from any subsequence of $(\mu_n)_{n\ge 1}$, we can extract a sub-subsequence that converges weakly to a probability measure $\nu\in\cP(\S')$. But, by weak convergence, the characteristic functional of $\nu$ is necessarily the pointwise limit of $(\cL_{\mu_n})_{n\ge 1}$, which is $\cL$. Hence $\nu=\mu$ and we conclude that $\mu_n$ converges weakly to $\mu$.

\medskip

\noindent It remains to prove that $(\mu_n)_{n\ge 1}$ is tight in $\cP(\S')$. This follows from the following two lemmas.
\end{proof}
\begin{lem}\label{lem:tight}
Let $(\mu_n)_{n\ge 1}$ be a sequence of Borel probability measures on $\S'$. If
$(\Re \cL_{\mu_n})_{n\ge 1}$ is equicontinuous on $\S$, then $(\mu_n)_{n\ge 1}$ is tight in $\cP(\S')$.
\end{lem}
\begin{proof}
Let $\eps>0$.
The equicontinuity of $(\Re \cL_{\mu_n})_{n\ge 1}$ implies that there exist $p\in\N$ and $\delta>0$ such that for all $n\ge1$, 
 \[
  \|a\|_p<\delta  \mbox{ implies }
   1-\Re \cL_{\mu_n}(a)<\eps.
\]
By Lemma~\ref{lem:bound} for $q=p+1$, for all $n\ge 1$ and all $\sigma>0$,
\[
 \int_{\S'}1- \exp\left(-\frac{\sigma^2}{2}\|b\|_{-p-1}^2\right)d\mu_n(b)\le \eps +2\delta^{-2}\sigma^2\zeta(2)^d.
\]
Then, fix $\sigma>0$ such that $2\delta^{-2}\sigma^2\zeta(2)^d=\eps$ and choose $\kappa>0$ such that $\exp\left(-\sigma^2\kappa^2/2\right)\le\eps$.
We need the fact that 
the set
\[
 K=\{b\in\S'\mid \|b\|_{-p-1}\le\kappa\}
\]
is strongly compact in $\S'$.
This follows from  the fact that in $\S'$, the strongly closed and bounded subsets are strongly compact. We provide a proof of this fact in  Proposition~\ref{prop:strongly_compact} in the appendix.
The set $\{b\in\S'\mid \|b\|_{-p}\le C\}$ is strongly (or weakly) closed as the intersection of the closed sets $\{b\in\S'\mid |\langle b,a\rangle|\le C\}$, $a\in\S$ with $\|a\|_p\le 1$. It is strongly bounded since for all bounded set $B\in\S$, with $M=\sup_{a\in B}\|a\|_p<\infty$, we have $q_B(b)\le M\|b\|_{-p}$.

We have, for all $n\ge1$,
\begin{align*}
\mu_n(K)
&\ge \int_{\S'}\exp\left(-\frac{\sigma^2}{2}\|b\|_{-p-1}^2\right)d\mu_n(b)- \int_{K^c}\exp\left(-\frac{\sigma^2}{2}\|b\|_{-p-1}^2\right)d\mu_n(b)\\
& \ge 1-2\eps-\eps.
\end{align*}
Thus $(\mu_n)_{n\ge 1}$ is tight.
\end{proof}
\begin{lem}$(\Re \cL_{\mu_n})_{n\ge 1}$  is equicontinuous. 
\end{lem}
\begin{proof}
 Assume that $(\Re \cL_{\mu_n})_{n\ge 1}$ is not equicontinuous. Then, there exist $\eps>0$ 
and
 a sequence of elements $a_k=(a_{k,n})_{n\in\N^d}\in\S$, $k\in\N$, such that $a_k\to 0$ in $\S$ as $k\to\infty$, and an increasing sequence of integers $(n_k)_{k\ge1}$ such that
\begin{equation}\label{eq:no-equicont}
 1-\Re \cL_{\mu_{n_k}}(a_k)\ge \eps, \quad \text{ for all }k\ge 1.
\end{equation}
Since $a_k\to 0$ in $\S$, considering a subsequence if necessary, we can assume that
$2^{k}a_k\to 0$ in $\S$. That is, for all $p\in\N$, $\lim_{k\to\infty}\|2^ka_k\|_p= 0$.
For a positive integer $k_0$ to be fixed later, we define the functional $Q:\S'\to \R_+$ by
\[
 Q(b)=\sum_{k\ge k_0}\langle b,a_k \rangle^2,\quad b\in\S'.
\]
Note that, if $b\in\S'$, there exists $p\ge k_0$ such that $b\in\S_{-p}$ and thus
\[
 Q(b)\le \sum_{k=k_0}^p\langle b,a_k \rangle^2+\|b\|_{-p}\sum_{k\ge p+1} \|a_k\|_p^2 <\infty
\]
because $\|a_k\|_p^2=o(2^{-2k})$. This proves that $Q$ is well-defined.
Further, $Q$ is continuous at $0$. Indeed, since $2^{k}a_k\to 0$, the set $B=\{2^ka_k\mid k\in\N\}$ is bounded in $\S$, and for all $\eta>0$, the neighborhood $V_\eta=\{b\in\S'\mid q_B(b)< \eta\}$ of $0$ in $\S'$ satisfies that for all $b\in V_\eta$, $|\langle b,a_k\rangle| \le \eta 2^{-k}$, and thus
\[
 |Q(b)-Q(0)|=Q(b)\le \eta^2 \sum_{k\ge k_0} 2^{-k}\le \eta^2.
\]
Next,
$Q$ is negative-definite, i.e.\ for all $m\in\N$, all $c_1,\ldots,c_m\in\C$ such that $\sum_{i=1 }^mc_i=0$, and all $b^1,\ldots,b^m\in \S'$,
\[
\summ i1m \sum_{j=1}^m c_i c_j Q(b^i-b^j)\le 0.
\]
Indeed, it is sufficient to notice that 
\[
\summ i1m\summ j1m c_i c_j Q(b^i-b^j)=-2\left(\sum_{i=1}^m c_i \sum_{k\ge k_0}\langle b^i,a_k \rangle\right)^2.
\]

\medskip

\noindent The functional $U:\S'\to [0,1]$, defined by 
\[
 U(b)=\exp(-Q(b)),\quad b\in\S',
\]
is then positive-definite (see \citep[Chap.~3, Theorem 2.2]{BerChrRes84}), continuous at $0$ and $U(0)=1$. 
Hence, by Theorem~\ref{thm:MinlosS} there exists a Borel probability measure $\nu$ on $\S$ such that $\cL_\nu=U$.

\medskip

\noindent Note that, since it can be written as the limit of the measurable maps $(b,a)\mapsto \sum_{n\in\Gamma_N}a_nb_n$ as $N\to\infty$, the bilinear map $\langle \cdot,\cdot\rangle$ is $\cB( \S')\times\cB(\S)$-measurable.
Using Fubini's theorem, for all $n\in\N$, we can write
\[
\int_{\S'} U(b) d\mu_n(b)
=\int_{\S'}\int_\S e^{i\langle b,a\rangle}d\nu(a)d\mu_n(b)
=\int_\S\cL_{\mu_n}(a)d\nu(a).
\]
But, by dominated convergence,  
\[
\limn \int_\S\cL_{\mu_n}(a)d\nu(a)=\int_\S\cL_\mu(a)d\nu(a).
\]
Again, by Fubini's theorem, $\int_\S\cL_\mu(a)d\nu(a)=\int_{\S'} U(b) d\mu(b)$, and thus,
\[
\limn \int_{\S'} U(b) d\mu_n(b)= \int_{\S'} U(b) d\mu(b).
\]
Now, recalling the definition of $U$, we can fix $k_0$ such that
\[
 \int_{\S'}1- U(b) d\mu(b)<\eps/2,
\]
and find $N\in\N$ such that for all $n\ge N$,
\[
 \int_{\S'}1- U(b) d\mu_n(b)<\eps.
\]
As a consequence, for all $n\ge N$, and $k\ge k_0$,
\[
 \int_{\S'}1- \exp\left(-\langle b,a_k\rangle^2\right) d\mu_n(b)\le \int_{\S'}1-U(b)d\mu_n(b)<\eps.
\]
We infer that for all $n\ge N$, and $k\ge k_0$,
\[
 1-\Re \cL_{\mu_n}(a_k)=\int_{\S'}1- \cos(-\langle b,a_k\rangle) d\mu_n(b)<M\eps,
\]
where $M=\sup_{x\ge 0}\frac{1-\cos(x)}{1-\exp(-x^2)}<\infty$.
This contradicts \eqref{eq:no-equicont}.
From Lemma~\ref{lem:tight}, we infer that $(\mu_n)_{n\ge 1}$ is tight in $\cP(\S')$.
\end{proof}

\section{Appendix}\label{sec:appendix}
The goal of the appendix is to prove  the following proposition. 
\begin{prop}\label{prop:strongly_compact}
In $\S'$, the strongly closed and strongly bounded subsets are strongly compact.
\end{prop}
This proposition tells that $\S'$ is a Montel space. See Remark~\ref{rem:Montel} at the end. We provide a self-contained proof here, which consists of several steps.
\begin{prop}\label{compactS}
In $\cS$ (and hence in $\S$), the bounded subsets are relatively compact.
\end{prop}
\begin{proof}
As $\cS$ is metrizable, compactness and sequential compactness are equivalent. Let $(f_n)_{n\in\N}$ be a bounded sequence in $\cS$.
Let $(K_\ell)_{\ell\ge 1}$ be an increasing sequence of compact subsets of $\R^d$ whose union is 
 $\R^d$.
Recall that $(\|\cdot\|_{m,j})_{m\in\N, j\in\N^d}$ in~\eqref{eq:norm_mj} and $(\|\cdot\|_p)_{p\in\N}$ in~\eqref{eq:nn_p} generate the same topology in $\cS$.
For 
each $m\in\N$ and $j\in\N^d$, the sequence of functions $x\to (1+|x|)^m D^jf_n(x)$, $n\in\N$, is uniformly bounded and equicontinuous on $K_1$ and hence, by Arz\`ela--Ascoli's theorem, there exists a subsequence that converges uniformly on $K_1$. By diagonal extraction, we can find a subsequence $(f_{\varphi(n)})_{n\in\N}$ for which $(1+|\cdot|)^m D^jf_{\varphi(n)}(\cdot)$ converges uniformly on $K_1$ for all $m\in\N$ and $j\in\N^d$. We can successively repeat the same procedure on each  $K_\ell$, $\ell\ge1$, and finally, another diagonal extraction leads to a subsequence that converges with respect to all the semi-norms $\nn\cdot_{m,j}$, and hence in $\cS$.
\end{proof}

\begin{lem}\label{lem:weakCompact}
 If $H\subset \S'$ is weakly closed and weakly bounded, then it is weakly compact.
\end{lem}
\begin{proof}
Let $H\subset \S'$ be weakly closed and weakly bounded. Then, for all $n\in\N^d$, the set 
$ H_n=\{b_n \mid (b_k)_{k\in\N^d}\in H\}$
is closed and bounded in $\R$. By Tychonoff's theorem, the set $\prod_{n\in\N^d}H_n$ is compact in $\R^{\N^d}$. Therefore $H$, being a closed subset of $\prod_{n\in\N^d}H_n$, is compact.
\end{proof}

We say that  a set $H\subset\S'$ is \textit{equicontinuous} if 
for all $\eps>0$, there exists a neighborhood $V$ of $0$ in $\S$ such that, for all $b\in H$ and $a\in V$, $|\la b , a\ra|<\eps$. 

\begin{lem}\label{lem:equivTopo}
 If $H\subset\S'$ is equicontinuous, the weak topology and the strong topology are identical on $H$.
\end{lem}

\begin{proof}
It suffices to prove that the induced strong topology on $H$ is weaker than the induced weak topology. For this purpose, we show that for all $\eps>0$, $B$ a bounded set in $\S$ and $b_0\in H$, 
 the strongly open set $O :=\{b\in H\mid q_B(b-b_0)<\eps\}$ 
 contains a weak neighborhood of $b_0$. 
By equicontinuity, there exists a neighborhood $V$ of $0$ in $\S$ such that 
$|\la b,a\ra|<\eps/3$
 for all $b\in H$ and $a\in V$.
Since $B$ is bounded, by Proposition~\ref{compactS}, it is relatively compact in $\S$. Thus, we can find $a_1,\ldots, a_n \in B$, such that
\[
 B\subset \bigcup_{i=1}^n(a_i+V). 
\]
The set $O':=\{b\in H\mid |\la b-b_0,a_i\ra|<\eps/3\text{ for all }i=1,\ldots,n\}$ is a weak neighborhood of $b_0$ and we have $O'\subset O$. Indeed, for all $b\in O'$   we can find, for all $a\in B$, $i\in\{1,\dots,n\}$, such that $a- a_i\in V$. Then,
\[
 |\la b-b_0,a\ra|\le |\la b,a-a_i\ra|+|\la b-b_0,a_i\ra|+|\la b_0,a-a_i\ra|< \eps/3+\eps/3+\eps/3.
\]
Thus, $q_B(b-b_0)<\eps$ and hence $b\in O$. 
\end{proof}

\begin{lem}\label{lem:equicont}
 If $H\subset\S'$ is strongly bounded, then it is equicontinuous.
\end{lem}
\begin{proof}
Let $H\subset\S'$ be strongly bounded. In the sequel, we denote $B_p(\eps)=\{a\in\S\mid \|a\|_p< \eps\}$ for $p\in\N$ and $\eps>0$.
We first show that the set $\{q_{B_{p_0}(1)}(b)\mid b\in H\}$ is bounded for some $p_0\in\N$.
Assume it is not the case. Then, for each $p\in\N$, there exist $a_p\in B_p(1/p)$ and $b_p\in H$ such that $|\la b_p,a_p\ra|>p$. Since each $a_p$ belongs to $B_p(1/p)$, the sequence $(a_p)_{p\in\N}$ converges to $0$ in $\S$ and thus it is bounded. But then, 
for the bounded set $B=\{a_p\mid p\in\N\}$, we have $q_B(b_p)>p$ for all $p\in\N$ which is a contradiction with the fact that $H$ is strongly bounded.\\
Now for such a $p_0$, letting $M=\sup_{b\in H} q_{B_{p_0}(1)}(b)<\infty$, we have that for all $\eps>0$ and for all $a\in B_{p_0}(\eps/M)$, 
\[
 |\la b,a\ra|\le \|b\|_{p_0}'\|a\|_{p_0}\le M\eps/M=\eps,\;\text{  for all }b\in H.
\]
\end{proof}

\begin{proof}[Proof of Proposition~\ref{prop:strongly_compact}]
Let $H$ be a strongly closed and strongly bounded subset of $\S'$. In particular, $H$ is weakly bounded and, by Lemma~\ref{lem:weakCompact}, its weak closure $\bar H$ is weakly compact.
Now, by Lemma~\ref{lem:equicont}, $H$ is equicontinuous and then $\bar H$ is also equicontinuous.
Hence, strong and weak topology are identical on $\bar H$ by Lemma~\ref{lem:equivTopo}. 
As $H$ is strongly closed, it is also weakly closed and we infer that $\bar H=H$. Hence $H$ is weakly compact and then strongly compact.
\end{proof}

\begin{rem}\label{rem:Montel}
Proposition~\ref{compactS} tells that $\cS$ and $\S$ are actually Montel spaces (i.e.\ a locally convex Hausdorff barreled space in which every closed bounded set is compact). See also
\citep[Example 10.IV and Proposition 34.4]{Tre67} 
and \citep[Section I.6.5]{gelfand68generalized}. Proposition~\ref{prop:strongly_compact} tells that  $\cS'$ and $\S'$ are also Montel spaces. Furthermore, as a consequence of Banach--Steinhauss theorem, any weakly convergent sequence of $\cS'$ (or $\S'$) is also strongly convergent (see \citep[Corollary 1 of Proposition 34.6]{Tre67}
and \citep[Section I.6.4]{gelfand68generalized}).
\end{rem}

\subsection*{Acknowledgements}
YW's research was partially supported by 
the NSA grants H98230-14-1-0318 and H98230-16-1-0322, the ARO grant W911NF-17-1-0006, and Charles Phelps Taft Research Center 
at University of Cincinnati.

\bibliographystyle{apalike}

\def\cprime{$'$} \def\polhk#1{\setbox0=\hbox{#1}{\ooalign{\hidewidth
  \lower1.5ex\hbox{`}\hidewidth\crcr\unhbox0}}}
  \def\polhk#1{\setbox0=\hbox{#1}{\ooalign{\hidewidth
  \lower1.5ex\hbox{`}\hidewidth\crcr\unhbox0}}} \def\cprime{$'$}
  \def\cprime{$'$} \def\cprime{$'$}

\end{document}